\documentclass{amsart}
\usepackage[utf8]{inputenc}

\title{Isotropy indices of Pfister multiples\\ in characteristic two}
\author{Nico Lorenz}
\address{{Fakult\"at f\"ur Mathematik}, {Ruhr-Universit\"at Bochum},{Universit\"atsstra{\ss}e 150}, {44780 Bochum}, {North Rhine-Westphalia}, {Germany}}
\email{nico.lorenz@ruhr-uni-bochum.de}

\author{Krist\'{y}na Zemkov\'{a}}
\address{Department of Mathematical and Statistical Sciences, University of Alberta, Edmonton T6G 2G1, Canada} 
\address{Department of Mathematics and Statistics, University of Victoria, Victoria BC V8W
2Y2, Canada}
\email{zemk.kr@gmail.com}
\date{\today}

\usepackage{amsthm}
\usepackage{amsmath}
\usepackage{amssymb}
\usepackage[shortlabels]{enumitem}
    \setlist[enumerate,1]{label=(\roman*), font=\normalfont}
    \setlist[enumerate,2]{label=(\alph*), font=\normalfont}
    \setlist{nosep}
\usepackage{tikz}
\usepackage{xcolor}
\usepackage[size=scriptsize]{todonotes}

\usepackage[
backend=bibtex,
style=alphabetic,
sorting=nyt, maxbibnames=99,isbn=false,url=false
]{biblatex}
\usepackage{comment}

\definecolor{Sepia}{HTML}{671800}
\definecolor{MidnightBlue}{HTML}{006795}
\usepackage[colorlinks=true]{hyperref}
\hypersetup{citecolor=Sepia, linkcolor=Sepia, urlcolor=MidnightBlue}

\usepackage{cleveref}

\usepackage{scalerel}
\newcommand{\Square}{\raisebox{-2pt}{\scaleobj{1.3}{\square}}}

\newtheorem{definition}{Definition}[section]
\newtheorem{question}[definition]{Question}
\newtheorem{theorem}[definition]{Theorem}
\newtheorem{lemma}[definition]{Lemma}
\newtheorem*{lemma*}{Lemma}
\newtheorem{corollary}[definition]{Corollary}

\newtheorem{remark}[definition]{Remark}
\newtheorem{example}[definition]{Example}
\newtheorem{proposition}[definition]{Proposition}

\numberwithin{equation}{section}

\newcommand{\N}{\mathbb{N}}
\newcommand{\F}{\mathbb{F}}

\newcommand{\squares}[1]{#1^2}

\newcommand{\spn}{\operatorname{span}}

\newcommand{\qf}[1]{\langle #1\rangle}

\newcommand{\bb}{\mathfrak{b}} 
\newcommand{\istar}[1]{\mathfrak{i}_{\ast}(#1)} 
\newcommand{\istarempty}{\mathfrak{i}_{\ast}} 
\newcommand{\iw}[1]{\mathfrak{i}_{\mathrm{W}}(#1)} 
\newcommand{\iwempty}{\mathfrak{i}_{\mathrm{W}}} 
\newcommand{\iql}[1]{\mathfrak{i}_{\mathrm{d}}(#1)} 
\newcommand{\iqlempty}{\mathfrak{i}_{\mathrm{d}}} 
\newcommand{\iti}[1]{\mathfrak{i}_{\mathrm{t}}(#1)} 
\newcommand{\itiempty}{\mathfrak{i}_{\mathrm{t}}} 
\newcommand{\ione}[1]{\mathfrak{i}_{1}(#1)}
\newcommand{\ioneempty}{\mathfrak{i}_{1}}
\newcommand{\itwo}[1]{\mathfrak{i}_{2}(#1)}
\newcommand{\ihigher}[2]{\mathfrak{i}_{#2}(#1)}
\newcommand{\an}{\mathrm{an}} 
\newcommand{\nd}{\mathrm{nd}} 
\newcommand{\ql}[1]{\mathrm{ql}\left(#1\right)}

\newcommand{\sqf}[1]{\langle #1 \rangle}
\newcommand{\spf}[1]{\langle\!\langle #1\rangle\!\rangle}

\newcommand{\bif}[1]{\langle #1 \rangle_b}
\newcommand{\bipf}[1]{\langle\!\langle #1\rangle\!\rangle_b}
\makeatletter

\newcommand{\bigperp}{%
  \mathop{\mathpalette\bigp@rp\relax}%
  \displaylimits
}
\newcommand{\bigp@rp}[2]{%
  \vcenter{
    \m@th\hbox{\scalebox{\ifx#1\displaystyle2.1\else1.5\fi}{$#1\perp$}}
  }%
}
\makeatother

\newcommand{\dbrac}[1]{(\!(#1)\!)}
\newcommand{\FX}{F\dbrac{X}}
\renewcommand{\H}{\mathbb{H}}
\newcommand{\new}[1]{#1}


\begin{document}
\begin{abstract}
    Let $F$ be a field of characteristic $2$, $\pi$ be an $n$-fold bilinear Pfister form over $F$ and $\varphi$ an arbitrary quadratic form over $F$.
    In this note, we investigate Witt index, defect, total isotropy index and higher isotropy indices of $\varphi$ and $\pi\otimes\varphi$ and prove relations among the indices of these two forms over certain field extensions.
\end{abstract}

\maketitle

\section{Introduction} \label{Sec:Intro}

Probably no class of quadratic forms is studied as extensively as the class of \emph{Pfister forms}. 
While occurring also in adjacent branches of mathematics as in Lagrange's four squares theorem or as norm forms of quaternion or octonion algebras, the understanding of Pfister forms has led to many deep results in the algebraic theory of quadratic forms.
The first breakthrough was the \emph{Arason-Pfister Hauptsatz} in \cite{ArasonPfister71}, whose proof uses the function field extension of a Pfister form. 
Pfister forms and the powers of the fundamental ideal connect the algebraic theory of quadratic forms to other areas like central simple algebras by A. Merkurjev's Theorem in \cite{MerkurjevI2Brauer}.
This connection was further enhanced to Milnor $K$-Theory and Galois cohomology by D. Orlov, A. Vishik and V. Voevodsky in \cite{Voevodsky, OrlovVishikVoevodsky} in the case of a field of characteristic not 2. 
In characteristic 2, these connections and further relations to differential forms were found earlier by K. Kato in \cite{Kat1}, S. Bloch, K. Kato in \cite{BlochKato}.
There are further plenty of structures arising from Pfister forms such as \emph{Pfister ideals} investigated, e.g., in \cite{ELWPfisterIdeals} and \cite{FitzgeraldWittKernels} or invariants as \emph{Pfister numbers} studied, e.g., in \cite{HoffmannTignol,Karpenko2017, LorenzPfisterNumbers}.

In his article \cite{OShea16}, J. O'Shea investigated multiples of Pfister forms over fields of characteristic not $2$, i.e., quadratic forms that have a decomposition $\pi\otimes\varphi$ with a bilinear Pfister form $\pi$ and a quadratic form $\varphi$. 
\new{His main focus is to link the isotropy behavior of $\varphi$ over the function field $F(\psi)$ with the one of $\pi\otimes\varphi$ over $F(\pi\otimes\psi)$, for an arbitrary quadratic form $\psi$. In particular, the choice $\psi\cong\varphi$ gives the first values in the standard splitting pattern of $\varphi$ resp. $\pi\otimes\varphi$; the \emph{standard splitting pattern} is an iterated sequence of isotropy indices of the given form over its own function field (see \Cref{Sec:Preliminaries}).
This sequence is of a major importance; the sequence itself and its connections to other invariants like the Clifford invariant has been investigated by numerous authors, see e.g. \cite{Knebusch1,Knebusch2,MR1241799,MR1361129,HoffmannSplitting}. Therefore, the ideal but unfulfilled goal of the paper \cite{OShea16}  is to link the standard splitting pattern of $\pi\otimes\varphi$ to the one of $\varphi$ (see Conjectures~1.1 and~1.2 of \cite{OShea16}).}

The aim of this paper is now to transfer the results from \cite{OShea16} to characteristic 2 whenever possible. There are technical subtleties that arise naturally, particularly because in characteristic 2, we have  \new{various types of quadratic forms with diverse behaviour. Consequently, we also have several} different indices that measure the degree of isotropy, namely the \emph{Witt index} $\iwempty$ and the \emph{defect} $\iqlempty$.
Further, their sum, the \emph{total isotropy index} $\itiempty$ and the \emph{first isotropy index} $\ioneempty$, i.e., the total isotropy index of a form over its function field are of interest.
Despite the differences in characteristic 2, some proofs work similarly as in \cite{OShea16} and from time to time, maybe after a reduction step, even the proofs of analogous results for different isotropy indices go along the same lines. 

The main questions considered in this article are thus of the following type:
\begin{question}\label{question:Intro}
    Given a bilinear Pfister form $\pi$ and a quadratic form $\varphi$ over some field $F$ of characteristic $2$ and $K, L$ field extensions of $F$ and an isotropy index $\istarempty\in\{\iwempty,\iqlempty, \itiempty, \ioneempty\}$, how are $\istar{(\pi\otimes\varphi)_K}$ and $\istar{\varphi_L}$ related to each other?
\end{question}

There are first results in the direction of \Cref{question:Intro} from the second author that we would like to recall:

\begin{proposition}[{\cite[Lemma~5.4]{KZ23-Isotropy}}]\label{Prop:IsotropyXXXPFmultiples}
Let $\varphi$, $\psi$ be nondefective quadratic forms over a field $F$ of characteristic 2 with $\dim\psi\geq2$. Let $\pi\cong\bipf{X_1,\dots,X_n}$ over $K=F\dbrac{X_1}\dots\dbrac{X_n}$. Then $\varphi$ is isotropic over $F(\psi)$ if and only if $\pi\otimes\varphi$ is isotropic over $K(\pi\otimes\psi)$.
\end{proposition}

\begin{theorem}[{\cite[Th.~5.5]{KZ23-Isotropy}}]\label{Th:IsotropyPfisterMultiples} 
Let $\varphi$, $\psi$ be quadratic forms over $F$, and let $\pi$ be a bilinear Pfister form over a field $F$ of characteristic 2. If $\varphi_{F(\psi)}$ is isotropic, then $(\pi\otimes\varphi)_{F(\pi\otimes\psi)}$ is isotropic.
\end{theorem}

An example which shows that the reverse implication in \Cref{Th:IsotropyPfisterMultiples} is not always true is further given in \cite[Remark~5.7]{KZ23-Isotropy}. In particular, we cannot always expect equality of the indices mentioned in \Cref{question:Intro} even for obvious choices of the fields $K, L$.

Note that these results only show the isotropy but do not specify the kind of isotropy, i.e., whether the Witt index or the defect is positive.
During this note, we will make these results more explicit. 
In fact, we will show in \Cref{Th:IoneIneq} the inequality $\ione{\pi\otimes\varphi}\geq(\dim\pi)\,\ione{\varphi}$. 
We will give examples with strict inequality in \Cref{ex:StrictInequalityIOne}.
We further give some sufficient conditions for equality in the rest of \Cref{sec:FirstIsotropyIndex}.

\Cref{sec:IndexOverFunctionFields} contains results about the isotropy over function field extensions of another quadratic form $\psi$. 
As a main result, we show that if $\psi$ is of dimension at least $2$, we have the inequality
\[\istar{(\pi\otimes\varphi)_{F(\pi\otimes\psi)}} \geq(\dim\pi)\,\istar{\varphi_{F(\psi)}}\]
for $\istarempty\in\{\iwempty,\iqlempty\}$ in \Cref{Cor:DefectsWittProducts}.

The last section is devoted to the study of multiples of generic Pfister forms, i.e., Pfister forms whose slots are all algebraically independent variables.
This special situation leads to examples for equalities in our main results, see \Cref{Th:GenericPfisterMultiple}.

\section{Preliminaries and Notation}\label{Sec:Preliminaries}

All upcoming fields are assumed to have characteristic 2. 
Further, vector spaces are all of finite dimension.

We write $\bif{a_1,\ldots, a_n}$ for a bilinear form on an $n$-dimensional vector space $V$ over $F$ that has the diagonal matrix with entries $a_1,\ldots, a_n$ as a Gram matrix for a suitable basis of $V$. Note that we have $\lambda\bif{a}\cong\bif{\lambda a}$ for any $\lambda\in F$.

A \emph{quadratic form} (or just a \emph{form} for short) over a field $F$ defined on some vector space $V$ is a map $\varphi:V\to F$ such that 
\begin{enumerate}
    \item for all $a\in F$ and $v\in V$, we have 
    \[\varphi(av)=a^2\varphi(v),\]
    \item the map 
    \begin{align*}
        \bb_\varphi:V\times V&\to F\\
        (v,w)&\mapsto \varphi(v+w)+\varphi(v)+\varphi(w)
    \end{align*}
    is a bilinear form.
\end{enumerate}

The \emph{dimension} of $\varphi$ is defined as $\dim\varphi:=\dim(V)$.

Let $a,b,c\in F$.
As usual, we denote by $[a,b]$ the quadratic form $F\times F\to F, (x,y)\mapsto ax^2+xy+by^2$ and by $\qf{c}$ the quadratic form $F\mapsto F, z\mapsto cz^2$. \new{Note that for any $\lambda\in F$, we have $\lambda\sqf{c}\cong\sqf{\lambda c}$, and if $\lambda\neq0$, then $\lambda[a,b]\cong [\lambda a,\lambda^{-1}b]$. Let $\varphi$ be a quadratic form of dimension $n$ on a vector space $V$ over $F$ and let $\psi$ be a quadratic form of dimension $m$ on a vector space $W$ over the same field $F$; then we denote by $\varphi\perp\psi$ the quadratic form $V\oplus W\to F$, $(x_1,\dots,x_n, y_1,\dots,y_m)\mapsto \varphi(x_1,\dots,x_n)+\psi(y_1,\dots,y_m)$. We write $\sqf{c_1,\dots,c_s}$ instead of $\sqf{c_1}\perp\dots\perp\sqf{c_s}$. Moreover, let $\bb\cong\bif{a_1,\dots,a_k}$ be a bilinear form; we define the tensor product of $\bb$ with $\varphi$ as
\[\bb\otimes\varphi\cong a_1\varphi\perp\dots \perp a_k\varphi.\]}

\new{While quadratic forms over fields of characteristic different from $2$ can always be diagonalised by \cite[Proposition 7.29]{ElmanKarpenkoMerkurjev2008}, the structure of quadratic forms over fields of characteristic 2 is slightly more involved and will be now recalled.}

For any quadratic form $\varphi$, we can find $a_1,\ldots, a_r,b_1,\ldots, b_r, c_1,\ldots, c_s\in F$ such that we have
\begin{align}\label{eq:DecQF}
    \varphi\cong[a_1,b_1]\perp\ldots\perp[a_r, b_r]\perp\qf{c_1,\ldots, c_s}.
\end{align}

The tuple $(r,s)\in\N\times\N$ is called the \emph{type} of $\varphi$ and is an invariant under isometry.
The form $\varphi$ is called
\begin{itemize}
    \item \emph{nonsingular}, if $s=0$,
    \item \emph{totally singular} (or \emph{quasilinear}), if $r=0$,
    \item \emph{singular}, if $s>0$.
\end{itemize}
The form $\qf{c_1,\ldots, c_s}$ in \eqref{eq:DecQF} is uniquely determined up to isometry and is called the \emph{quasilinear part} of $\varphi$ and denoted by $\ql{\varphi}$.

By $D_F(\varphi)=\{\varphi(v)\mid v\in V,\, \varphi(v)\neq0\}$, we denote the set of \emph{represented elements} of $\varphi$ and by $G_F(\varphi)=\{a\in F^\ast\mid a\varphi\cong\varphi\}$ the set of \emph{similarity factors}.

A quadratic form $\varphi$ is called \emph{isotropic} if it represents 0 nontrivially, i.e., if there is some $v\in V\setminus\{0\}$ with $\varphi(v)=0$ and \emph{anisotropic} otherwise.
Since we consider fields of characteristic 2, there are two typical examples of isotropic forms, the \emph{hyperbolic plane} $\H=[0,0]$ and the quadratic form $\qf0$.

By Witt's decomposition theorem we can refine the decomposition in \eqref{eq:DecQF}, and write any quadratic form $\varphi$ as

\begin{align}\label{eq:WittDec}
    \varphi\cong i\times\H\perp\varphi_r\perp\varphi_s\perp j\times\qf0
\end{align}
with $\varphi_r$ nonsingular, $\varphi_s$ totally singular, and $\varphi_r\perp\varphi_s$ anisotropic. 

In this decomposition, the integers $i,j$ are uniquely determined.
We define
\begin{itemize}
    \item the \emph{Witt index} of $\varphi$ as $\iw{\varphi}:=i$,
    \item the \emph{defect} of $\varphi$ as $\iql{\varphi}:=j$,
    \item the \emph{total isotropy index} of $\varphi$ as $\iti{\varphi}:=i+j$. 
\end{itemize}
A quadratic form $\varphi$ is thus anisotropic if and only if $\iti{\varphi}=0$.

We call $\varphi$ \emph{nondefective} if $\iql{\varphi}=0$.
In \eqref{eq:WittDec}, the forms $\varphi_\an=\varphi_r\perp\varphi_s$ resp. $\varphi_\nd=i\times\H\perp\varphi_r\perp\varphi_s$ are unique up to isometry and called the \emph{anisotropic part} resp. the \emph{nondefective part} of $\varphi$. 
We define two forms $\varphi,\psi$ to be \emph{Witt equivalent} if $\varphi_\an\cong\psi_\an$.\\

In order to construct examples, we will often use purely transcendental field extensions or Laurent series extensions.
Over these extensions, the isotropy indices can be controlled as the following lemma shows.

\begin{lemma} \label{Lemma:IndicesResidueForms}
 Let $\varphi_0$ and $\varphi_1$ be quadratic forms over $F$, and let $\varphi\cong\varphi_0\perp X\varphi_1$ over $F\dbrac{X}$ or $F(X)$. 
 Then 
 \[\iw{\varphi}=\iw{\varphi_0}+\iw{\varphi_1} \quad \text{and} \quad \iql{\varphi}=\iql{\varphi_0}+\iql{\varphi_1}.\]
 In particular, $\varphi$ is anisotropic if and only if $\varphi_0$ and $\varphi_1$ are anisotropic.
\end{lemma}

\begin{proof}
For $i\in\{0,1\}$, write
\[\varphi_i\cong\iw{\varphi_i}\times\H \perp \varphi_i'\perp\iql{\varphi_i}\times\sqf{0}.\]  
By \cite[Lemma~19.5]{ElmanKarpenkoMerkurjev2008}, the form $\varphi'\cong\varphi_0'\perp X\varphi_1'$ is anisotropic over $\FX$. 
Hence, $\varphi'\cong\varphi_\an$, and the claim follows.   
\end{proof}

\new{Note that the above proof for the Witt index does not depend on the characteristic of $F$.}

Let $\psi$ be another quadratic form over $F$. 
We say that $\psi$ is dominated by $\varphi$, denoted by $\psi\preccurlyeq\varphi$, if $\psi\cong\varphi|_U$ for some vector space $U\subseteq V$. 
Moreover, we say that $\psi$ is a subform of $\varphi$, denoted by $\psi\subseteq\varphi$, if there exists a quadratic form $\sigma$ such that $\varphi\cong\sigma\perp\psi$. 
Note that any subform of $\varphi$ is dominated by $\varphi$. 
If $\psi$ is nonsigular or if $\varphi$ is totally singular, then the notions of subform and dominance are equivalent, but not so in general: 
For example, $\qf0\preccurlyeq\H$, but $\qf0\perp\qf a\not\cong\H$ for any $a \in F$, because isometry preserves the type.

We have the following criterion for isotropy of dominated forms using the total isotropy index, \new{representing a characteristic 2 version of \cite[Chapter I. Exercise 16 (3)]{Lam2005}}.

\begin{lemma} [{\cite[Lemma 2.11]{HL06}}] \label{Lem:SubFormOfIsotropicFormIsIsotropic}
    If $\psi\preccurlyeq \varphi$ with $\dim\psi\geq\dim\varphi-\iti{\varphi} + 1$, then $\psi$ is isotropic.
\end{lemma}

Let $\bb$ be a bilinear form defined over a vector space $V$. The set of \emph{represented elements} of $\bb$ is given by $D_F(\bb)=\{\bb(v,v)\mid v\in V\}$ and the set of \emph{similarity factors} of $\bb$ is given by $G_F(\bb)=\{a\in F^\ast\mid a\bb\cong\bb\}$.
If the field is clear from the context, we will drop the subscript to shorten the notation.
A bilinear form $\bb$ is called \emph{round} if we have $D_F(\bb)=G_F(\bb)$.

As stated in the introduction, we have a special interest in quadratic forms that can be written as a tensor product of a quadratic form with a bilinear Pfister form.
A \emph{$1$-fold bilinear Pfister form} is a bilinear form isometric to $\bipf a:=\bif{1,a}$ for some $a\in F^\ast$ and for $n\geq2$, an \emph{$n$-fold bilinear Pfister form} is a form isometric to one of the shape
\[\bipf{a_1,\ldots, a_n}:=\bipf{a_1}\otimes\ldots\otimes\bipf{a_n}\]
for some $a_1,\ldots, a_n\in F^\ast$.
By convention, $\bif1$ is considered as a $0$-fold bilinear Pfister form.
It is well known that bilinear Pfister forms are round \cite[Corollary~6.2]{ElmanKarpenkoMerkurjev2008}.\\

A \emph{$1$-fold quadratic Pfister form} is a form isometric to $[1,a]$ for some $a\in F$ and for $n\geq2$ an \emph{$n$-fold quadratic Pfister form} is form isometric to $\pi\otimes[1,a]$ for some $(n-1)$-fold bilinear Pfister form $\pi$.
An \emph{$n$-fold quasi-Pfister form} is a form isometric to $\pi\otimes\qf1$ for some $n$-fold bilinear Pfister form $\pi$.

If $E/F$ is a field extension and $\varphi$ a quadratic form over $F$ defined on $V$, we denote by $\varphi_E$ the quadratic form obtained by scalar extension to $E$. 
Explicitly, $\varphi_E$ is defined on $V_E:=E\otimes_FV$ and for $a,b\in E, v,w\in V$, we have
\[\varphi_E(a\otimes v)=a^2\varphi(v)\text{ and }\bb_{\varphi_E}(a\otimes v,b\otimes w)=ab\bb_\varphi(v,w).\]

\new{The way the isotropy indices behave under field extensions give insight on the structure of a quadratic form.
Of particular interest is the case of quadratic field extensions. 
Recall that in characteristic 2, there are two types of quadratic extensions: inseparable, i.e., of the form $F(\sqrt{d})/F$, and separable, i.e., of the form $F(\wp^{-1}(d))/F$, where  
\[\wp:F\to F,\ x\mapsto x^2+x\]
is the \emph{Artin-Schreier map}. The necessity of this distinction makes the following proposition more involved than its characteristic not 2 analogue, where only separable field extensions can occur (see  \cite[Corollary 22.12]{ElmanKarpenkoMerkurjev2008}).
}

\begin{proposition}[{\cite[Prop.~2.11]{Lag06}, \cite[Prop.~1.1 and Rem.~1.2(i)]{Ahmad04}}] \label{Prop:IsotropySubforms}
Let $\varphi$ be an anisotropic quadratic form over $F$ and $d\in F^*$.
\begin{enumerate}
    \item \label{Prop:IsotropySubforms1} Let $\varphi$ be totally singular. Then there exists a totally singular quadratic form $\tau$ over $F$ of dimension $\iql{\varphi_{F(\sqrt{d})}}$ such that $\tau\otimes\sqf{1,d}\subseteq\varphi$.
    \item\label{Prop:IsotropySubforms2} Let $\varphi$ be nonsingular. Let $r=\iw{\varphi_{F(\sqrt{d})}}$. Then there exist $a_i,b_i,c_i\in F^*$, $1\leq i\leq r$, such that $\bigperp_{i=0}^{r}([a_i,b_i]\perp d[a_i,c_i])\subseteq\varphi$.
    \item \label{Prop:IsotropySubforms3} Let $\varphi$ be nonsingular. Then there exists a bilinear form $\tau$ over $F$ of dimension $\iw{\varphi_{F(\wp^{-1}(d))}}$ such that $\tau\otimes[1,d]\subseteq\varphi$.
\end{enumerate}
\end{proposition}

\new{Note that the case of a totally singular form over the field $F(\wp^{-1}(d))$ is missing in \Cref{Prop:IsotropySubforms}.  That is because } it is well known that anisotropic totally singular forms remain anisotropic over separable field extensions by \cite[Lemma 2.8]{Lag06}.

A special emphasis will further lie on function field extensions, which now will be introduced. 
For a quadratic form $\varphi$ of dimension $n$, the polynomial $\varphi(X_1,\ldots , X_n)$ is reducible if and only if the nondefective part $\varphi_\nd$ of $\varphi$ is either of the type $(0, 1)$, or of the type $(1, 0)$ and $\varphi_\nd \cong \H$, see \cite[Remark 1]{Ahmad1997}). 
We say that the quadratic form $\varphi$ is \emph{irreducible} if the polynomial $\varphi(X_1,\ldots,X_n)$ is irreducible; in this case, $\varphi(1,X_2,\dots,X_n)$ is also irreducible, and we can define the \emph{function field} $F(\varphi)$ as the quotient field of the integral domain $F[X_2,\ldots, X_n]/(\varphi(1,X_2,\ldots, X_n))$. For example, if $[1,d]$ is anisotropic, we have $F([1,d])\cong F(\wp^{-1}(d))$.
More generally, the extension $F(\varphi)/F$ is separable if $\varphi$ is not totally singular by \cite[Lemma 22.14]{ElmanKarpenkoMerkurjev2008}.
If $\varphi\cong\qf{a}$ or $\varphi\cong\H$, we define $F(\varphi)=F$.

Recall that Witt index and defect of a quadratic form do not change over a purely transcendental field extension (see \cite[Lemma~7.15]{ElmanKarpenkoMerkurjev2008}). 
Thus, the following lemma will often help us to reduce the problem to anisotropic quadratic forms.

\begin{lemma} \label{Lemma:FunFieldPurTransc}
Let $\psi$ be an irreducible quadratic form over $F$.
\begin{enumerate}
    \item \label{Lemma:FunFieldPurTransc1} 
     The field extension $F(\psi)/F(\psi_\nd)$ is purely transcendental.
    \item \label{Lemma:FunFieldPurTransc2} 
    If $\iw{\psi}>0$, then the field extension $F(\psi)/F$ is purely transcendental.
\end{enumerate} 
\end{lemma}

\begin{proof}
Part \ref{Lemma:FunFieldPurTransc1} can be found in \cite[Section~4.2]{HL04}. Part \ref{Lemma:FunFieldPurTransc2} follows from \ref{Lemma:FunFieldPurTransc1}  and \cite[Prop.~22.9]{ElmanKarpenkoMerkurjev2008}.
\end{proof}

The \emph{standard splitting tower} of a quadratic form $\varphi$ defined over a field $F$ is constructed as follows:
We set $F_0:=F$ and $\varphi_0:=\varphi_\an$. 
For $n\geq1$, we define $F_n:=F(\varphi_{n-1})$ and $\varphi_n:=(\varphi_{F_n})_\an$ as long as $\dim\varphi_n>1$. 
The form $\varphi_n$ is called the \emph{$n$-th kernel form} of $\varphi$ and the least integer $h$ with $\dim\varphi_h\leq1$ is called the \emph{height} of $\varphi$.
In characteristic 2, these concepts were first investigated in \cite{LaghribiGenericSplitting} and a systematic study for totally singular forms can be found in \cite{Lag04}.

For $n\geq1$ we define the \emph{$n$-th isotropy index} as $\ihigher{\varphi}{n}:=\iti{\varphi_{F_n}}$.
Let now $\varphi$ be of type $(r,s)$ with $r\geq1$.
By \cite[Theorem 4.6]{LaghribiGenericSplitting} there is some integer $1\leq i\leq h$ such that for $j\leq i$, the form $\varphi_j$ is of type $(r_j, s)$ with $r_j<r_{j-1}$ and $r_j=0$ for $j\geq i$.
This integer $i$ is called the \emph{nondefective height} of $\varphi$ and is denoted by $h_\nd(\varphi)$. 
In particular, if $\varphi_0$ is not totally singular, we have $\ione{\varphi}=\iw{\varphi_{F_1}}$ and $\ql{\varphi_j}=\ql{\varphi}_{F_j}$ for all $j\leq h_\nd(\varphi)$. 

If $E/F$ is any field extension and $(\varphi_E)_\an$ is of type $(r', s')$, then by \cite[Proposition 4.6]{HL06}, there is some $j\leq h_\nd(\varphi)$ such that $\varphi_j$ is of type $(r', s)$. 

The above combined with \cite[Corollary 4.8]{Scu16-Split} to cover the case of totally singular forms now yields the following property of the first isotropy index. 

\begin{corollary}\label{cor:FirstWittIndexMinimal}
    Let $\varphi$ be an anisotropic quadratic form over $F$ and $E/F$ a field extension such that $\varphi$ is isotropic and, if $\varphi$ is not totally singular, $\iw{\varphi_E}>0$.
    We then have
    \[\ione{\varphi}\leq\iti{\varphi_E}.\]
\end{corollary}

\new{Recall that for nonsingular forms of characteristic not 2, there is only one type of isotropy, measured by the Witt index.
Therefore the characteristic not 2 version of the above Corollary obtained by M. Knebusch in \cite[Proposition 3.1, Theorem 3.3]{Knebusch1} just uses Witt indices in its formulation.}

To conclude this introductory section, we will finally illuminate another aspect of function field extensions.
It is a longstanding problem in the algebraic theory of quadratic forms to determine which quadratic forms become isotropic or hyperbolic over the function field of a given form, see, e.g., \cite{Hoffmann6DimFunctionField, HoffmannIsotropyOverFunctionFields, IzhboldinIsotropyFunctionField, IzhboldinKarpenkoIsotropyDim6, LaghribiFunctionFieldDim6, HL06}. 
We have a necessary condition using the first isotropy index that was first shown for fields of characteristic not 2 by N. Karpenko and A. Merkurjev in \cite[Theorem 4.1]{KarpenkoMerkurjevEssDim} and then transferred to the characteristic 2 case by B.~Totaro.

\begin{theorem}[{\cite[Th.~5.2]{Totaro08}}] \label{IoneIsotropyEquivalence}
Let $\varphi$ and $\psi$ be anisotropic quadratic forms over $F$. If $\varphi_{F(\psi)}$ is isotropic, then
\[\dim\psi-\ione{\psi}\leq\dim\varphi-\ione{\varphi}\]
with equality if and only if $\psi_{F(\varphi)}$ is isotropic as well.
\end{theorem}

\new{As another necessary condition, we have the following characteristic 2 version of Hoffmann's Separation Theorem \cite[Theorem 1]{HoffmannIsotropyOverFunctionFields}.
It was transferred to characteristic 2 by D. Hoffmann itself in collaboration with A. Laghribi and states that the dimensions of two forms such that one of them becomes isotropic over the function field of the other cannot be separated by a power of two:}

\begin{theorem}[{Hoffmann's Separation Theorem, \cite[Theorem 1.1]{HL06}}]\label{thm:HoffmannSeparation}
    Let $F$ be a field of characteristic $2$. Let $\varphi,\psi$ be anisotropic quadratic forms over $F$ (possibly singular) such that $\dim\varphi \leq 2^n < \dim\psi$ for some integer $n \geq 1$. 
    Then $\varphi$ stays anisotropic over $F(\psi)$.
\end{theorem}


\section{The first isotropy index}\label{sec:FirstIsotropyIndex}

\subsection*{Main Inequality for the First Isotropy Index}

\new{The aim of this subsection is to prove that $\ione{\pi\otimes\varphi}\geq(\dim\pi)\,\ione{\varphi}$. In order to do that, we first need a characteristic $2$ version of a classical result stating that if $\varphi$ is divisible by a round form $\pi$, then both $\varphi_\an$ and the hyperbolic part of $\varphi$ are divisible by $\pi$, and hence $\dim\pi$ divides $\iw{\varphi}$; see \cite[Th.~2]{WadsShap77}.  
}

\begin{proposition} \label{Prop:DivideFactor}
    Let $\pi$ be an anisotropic round bilinear form  and $\varphi$ be a quadratic form of type $(r,s)$. 
    If $\pi\otimes \varphi$ is isotropic, then there exists an anisotropic form $\varphi_1$ and integers $j, k\in\N$ such that 
    \[\pi\otimes\varphi \cong \pi\otimes(\varphi_1\perp k\times\H\perp j\times \qf0),\]
    $\pi\otimes\varphi_1$ is anisotropic and $\varphi_1\perp k\times\H\perp j\times \qf0$ is of type $(r,s)$.
    In particular, we have
    \[\iw{\pi\otimes\varphi} = k\dim\pi\text{ and }\iql{\pi\otimes\varphi} = j\dim\pi.\]
\end{proposition}

\begin{proof}
    There are $a_1,\ldots, a_r,b_1,\ldots, b_r, c_1,\ldots, c_s\in F$ with \[\varphi\cong[a_1,b_1] \perp\ldots\perp[a_r,b_r]\perp\qf{c_1,\ldots, c_s}.\]
    Since
    \begin{align*}
        \pi\otimes\varphi &\cong \pi\otimes([a_1,b_1] \perp\ldots\perp[a_r,b_r]\perp\qf{c_1,\ldots, c_s})\\
        &\cong \bigperp_{i=1}^r\pi\otimes[a_i,b_i]\perp\bigperp_{j=1}^s\pi\otimes\qf{c_j}
    \end{align*}
    is isotropic, there are $y_1,\ldots, y_r,z_1,\ldots, z_s\in D(\pi)=G(\pi), \alpha_i\in D([a_i,b_i])$ and $x_1,\ldots, x_s\in F$ such that
    \[\sum\limits_{i=1}^r y_i\alpha_i+\sum\limits_{j=1}^sz_jc_jx_j^2=0.\]
    The form
    \[\rho= y_1[a_1,b_1] \perp\ldots\perp y_r[a_r,b_r]\perp\qf{z_1c_1,\ldots, z_sc_s}\] 
    is then clearly isotropic and of the same type as $\varphi$. 
    Using that $\pi$ is round, we readily see the isometry
    \[\pi\otimes\rho\cong\pi\otimes\varphi.\]
    If $\pi\otimes\rho_{\an}$ is anisotropic, we are done.
    Otherwise, we can apply the same procedure to $\rho_{\an}$ instead of $\varphi$ and will eventually come to the conclusion.
\end{proof}

\begin{corollary}\label{Cor:EasyEstimateForIone}
    Let $\varphi$ be a form of dimension at least 2 and let $\pi$ be an anisotropic bilinear Pfister form such that $\pi\otimes\varphi$ is anisotropic.
    Then $\ione{\pi\otimes\varphi}\geq\dim\pi$.
\end{corollary}

\begin{proof}
    We have 
    \[\ione{\pi\otimes \varphi}=\iti{(\pi\otimes \varphi)_{F(\pi\otimes \varphi)}}=\iti{\pi_{F(\pi\otimes \varphi)}\otimes \varphi_{F(\pi\otimes \varphi)}}.\]
    The form $\pi_{F(\pi\otimes \varphi)}$ is a bilinear Pfister form and thus round. 
    It is further anisotropic by \Cref{thm:HoffmannSeparation} and thus, the assertion follows from \Cref{Prop:DivideFactor}.
\end{proof}

\begin{theorem} \label{Th:IoneIneq}
    Let $\varphi$ be a form of dimension at least 2 and $\pi$ be similar to an anisotropic bilinear Pfister form such that $\pi\otimes\varphi$ is anisotropic over F.
    We then have
    \[\ione{\pi\otimes\varphi}\geq (\dim\pi)\,\ione\varphi.\]
\end{theorem}

\begin{proof}
    If $\ione{\varphi}=1$, the assertion follows from \Cref{Cor:EasyEstimateForIone}.
    Let thus now $\ione{\varphi}>1$.
    Let $\varphi'\preccurlyeq \varphi$  be a form of dimension $\dim\varphi-\ione{\varphi}+1$. 
    Then, by \Cref{Lem:SubFormOfIsotropicFormIsIsotropic}, $\varphi'_{F(\varphi)}$ is isotropic.
    By \Cref{Th:IsotropyPfisterMultiples}, $\pi\otimes\varphi'$ is isotropic over $F(\pi\otimes\varphi)$.
    Since $\pi\otimes\varphi$ is anisotropic over $F$ by assumption and $\pi\otimes\varphi'\preccurlyeq \pi\otimes\varphi$, we have that $\pi\otimes\varphi'$ is anisotropic over $F$.
    Further, $\pi\otimes\varphi$ is clearly isotropic over $F(\pi\otimes\varphi')$ and thus, \Cref{IoneIsotropyEquivalence} implies
    \[\dim(\pi\otimes\varphi)-\ione{\pi\otimes\varphi}= \dim(\pi\otimes\varphi')- \ione{\pi\otimes\varphi'},\]
    which is equivalent to
    \begin{align*}
        \ione{\pi\otimes\varphi}&= \ione{\pi\otimes\varphi'} +(\dim\pi)\,(\dim\varphi-\dim\varphi')\\
        &=\ione{\pi\otimes\varphi'} +(\dim\pi)\,(\ione{\varphi}-1).
    \end{align*}
    Since $\ione{\pi\otimes\varphi'}\geq\dim\pi$ by \Cref{Cor:EasyEstimateForIone}, the assertion follows.
\end{proof}

Later, in \Cref{Th:GenericPfisterMultiple}, we will see a situation in which we have equality in \Cref{Th:IoneIneq}.
The following example shows that also a strict inequality is possible.

\begin{example}\label{ex:StrictInequalityIOne}
    Let $F=\F_2(a,b,t)$ with algebraically independent indeterminates $a,b,t$.
    \begin{enumerate}
    \item
    Let $\varphi=[1, a]\perp t[1, b]$. 
    By \Cref{Lemma:IndicesResidueForms}, $\varphi$ is anisotropic.
    We clearly have $\ione{\varphi}\in\{1,2\}$.
    Since $a,b$ are independent indeterminates, it is clear that $\varphi$ is not similar to a 2-fold Pfister form and thus $\ione{\varphi}=1$ by \cite[Theorem~4.2(ii)]{HL06}.
    
    We further consider the anisotropic bilinear Pfister form $\pi=\bipf{a+b}$.  
    The forms
    \[\bipf{a+b}\otimes[1,a]\perp\bipf{a+b}\otimes[1,b]\text{ and }\bipf{a+b}\otimes[1,a+b]\]
    are Witt equivalent by \cite[Lemma 15.1]{ElmanKarpenkoMerkurjev2008}.
    Since the latter form is clearly isotropic, hence hyperbolic as a quadratic Pfister form, we conclude that $\bipf{a+b}\otimes[1,a]$ and $\bipf{a+b}\otimes[1,b]$ are isometric.
    Clearly, these forms are anisotropic by \Cref{Lemma:IndicesResidueForms} applied to $\F_2(b)(X)$ with $X=a+b$.
    Thus
    \begin{align*}
        \pi\otimes\varphi=\pi\otimes[1,a]\perp t\pi\otimes[1,b]\cong \bipf{t, a+b}\otimes[1,a]
    \end{align*}
    is an anisotropic 3-fold quadratic Pfister form.
    Its first isotropy index is thus given by 
    \[\ione{\pi\otimes\varphi}=4>2\cdot 1=(\dim\pi)\,\ione{\varphi}.\]
    \item  For an example in the case of totally singular quadratic forms, consider $\varphi\cong\sqf{1,a,t,abt}$ and $\pi\cong\bipf{b}$. 
    Then $\pi\otimes\varphi\cong\spf{a,b,t}$ is a 3-fold quasi-Pfister form, and hence $\ione{\pi\otimes\varphi}=4$ by \cite[Proposition 3.3(2)(i)]{Lag04}. 
    On the other hand, we have $1\leq\ione{\varphi}\leq\iti{\varphi_L}=\iql{\varphi_L}$ for any $L/F$ such that $\varphi_L$ is isotropic by \Cref{cor:FirstWittIndexMinimal}; in particular, $\ione{\varphi}\leq\iql{\varphi_{F(\sqrt{a})}}=1$, so $\ione{\varphi}=1$. 
    Hence, we have
    \[\ione{\pi\otimes\varphi}=4>2\cdot1=(\dim\pi)\,\ione{\varphi}.\]
    \end{enumerate}
\end{example}

As a consequence of \Cref{Th:IoneIneq}, we get a sufficient condition on a form dominated by $\pi\otimes\varphi$ to become isotropic over $F(\pi\otimes\varphi)$.

\begin{corollary}\label{cor:BigSubformsOfProductsBecomeIsotropic}
    Let $\varphi$ a form of dimension at least two and $\pi$ similar to a bilinear Pfister form be such that $\pi\otimes\varphi$ is anisotropic over $F$. 
    If $\psi\preccurlyeq \pi\otimes\varphi$ is a dominated form over F of codimension less than $(\dim\pi)\,\ione{\varphi} $, then $\psi$ is isotropic over $F(\pi\otimes\varphi)$.
\end{corollary}
\begin{proof}
    By \Cref{Th:IoneIneq}, we have $(\dim\pi)\,\ione \varphi\leq\ione{\pi\otimes\varphi}$ and thus 
    \[\dim\psi\geq \dim(\pi\otimes\varphi)-\ione{\pi\otimes\varphi}+1.\]
    Now, \Cref{Lem:SubFormOfIsotropicFormIsIsotropic} implies $\psi$ to be isotropic over $F(\pi\otimes\varphi)$.
\end{proof}

\subsection*{Conditions for Equality}

After having established a general inequality between the first isotropy indices of a quadratic form $\varphi$ and its Pfister multiples, we now look for situations in which this inequality is fulfilled with equality.

\new{In the following proposition, we start with a condition that includes the standard splitting tower of $\varphi$. 
Additionally, we also look at the second isotropy index. This proposition is the characteristic 2 counterpart of \cite[Prop.~2.6]{OShea16}. 
Note that in part \ref{item:HigherTotalIndices2}, we need to assume that $\pi\otimes\varphi$ is not a quadratic Pfister form, as otherwise the first kernel form $(\pi\otimes\varphi)_1$ is the zero form and $\itwo{\pi\otimes\varphi}$ is not defined. 
Moreover, note that if $\varphi$ is a quadratic Pfister form, then $\pi\otimes\varphi$ is also a quadratic Pfister form. 
Therefore, we need to assume that $\varphi$ is of height at least two. 
Unfortunately, for our proof, we need to know not only that $\varphi_1$ is nonzero, but also that it is not totally singular.
The reason is that the higher isotropy indices of totally singular forms are less ``well-behaved", in that not every anisotropic dimension over a field extension also occurs as the dimension of some kernel form, see \cite[Example 8.15]{HL04}.
Thus the behavior of $\varphi_2$ would not be predictable if $\varphi_1$ was totally singular; therefore, we assume that the \emph{nondefective height} of $\varphi$ is at least two.}

\begin{proposition}\label{prop:HigherWittIndexOfPfisterMultiples}
    Let $\varphi$ be a form of dimension at least two and $\pi$ similar to a bilinear Pfister form be such that $\pi\otimes\varphi$ is anisotropic over $F$. 
    Denote by ${F = F_0 \subseteq F_1 \subseteq \ldots \subseteq F_h}$ the standard splitting tower of $\varphi$ and $\varphi_0\cong\varphi,\varphi_1,\ldots,\varphi_h$ the kernel forms.
    \begin{enumerate}
        \item\label{item:HigherTotalIndices1} If $\pi\otimes \varphi_{1}$ is anisotropic over $F_1$, then $\ione{\pi\otimes\varphi} = (\dim\pi)\,\ione \varphi $.
        \item\label{item:HigherTotalIndices2} 
        Suppose that $\varphi$ is a quadratic form of nondefective height at least 2 and let $\pi\otimes\varphi$ be not similar to a quadratic Pfister form.
        If $\ione{\pi\otimes\varphi} = (\dim \pi)\, \ione \varphi$ holds and $\pi\otimes\varphi_{2}$ is anisotropic over $F_2$, then $\itwo{\pi\otimes\varphi} \geq (\dim \pi)\,\itwo{\varphi}$.
    \end{enumerate}
\end{proposition}
\begin{proof}
    \ref{item:HigherTotalIndices1}: By \Cref{Th:IoneIneq}, we have $\ione{\pi\otimes\varphi}\geq \ione\varphi\cdot \dim\pi$.
    On the other hand, over $F_1$, we write $\varphi_{F_1}\cong i_1\times\H\perp j_1\times\qf0\perp \varphi_1$, with $i_1+j_1=\ione \varphi$.
    The form 
    \begin{align*}
        (\pi\otimes\varphi)_{F_1} \cong\pi_{F_1} \otimes (i_1\times\H\perp j_1\times\qf0\perp \varphi_1)
    \end{align*}
    clearly has total isotropy index at least $(\dim\pi)\,\ione{\varphi}$ and since $\pi\otimes\varphi_1$ is anisotropic over $F_1$, we even have equality. 

    If $\varphi$ is not totally singular, then so is $\pi\otimes\varphi$.
    Then $\iw{\varphi_{F_1}}\geq1$ and thus, $\iw{(\pi\otimes\varphi)_{F_1}}\geq1$.
    We can thus apply \Cref{cor:FirstWittIndexMinimal} in the case of a form that is not totally singular and clearly also for totally singular forms and we  thus have
    \[\ione{\pi\otimes\varphi}\leq \iti{(\pi\otimes\varphi)_{F_1}}=(\dim\pi)\,\ione{\varphi}.\]
    
    \ref{item:HigherTotalIndices2}: 
    Let $(r,s)$ be the type of $\varphi$.
    Then, since $h_\nd(\varphi)\geq2$, the type of $\varphi_1$ resp. $\varphi_2$ are given by $(r_1, s)$ resp. $(r_2, s)$ for some $0\leq r_2 < r_1 < r$.
    The type of $\pi\otimes\varphi$ is clearly given by $(r\dim\pi, s\dim\pi)$ and since $\ione{\pi\otimes\varphi} = (\dim \pi)\, \ione \varphi$, we know that the type of $((\pi\otimes\varphi)_{F(\pi\otimes\varphi)})_\an$ is given by $(r_1\dim\pi, s\dim\pi)$.
    In particular, this form is not totally singular.

    Over $F_2$, we write $\varphi_{F_2}=i_2\times\H\perp\varphi_2$.
    We then have
    \begin{align*}
        (\pi \otimes\varphi )_{F_2} \cong \pi_{F_2} \otimes (i_2\times\H\perp\varphi_2).
    \end{align*}
    Since by assumption, $(\pi\otimes\varphi_2)_{F_2}$ is anisotropic, we have 
    \begin{align*}
        \iw{(\pi\otimes\varphi)_{F_2}}&= \iti{(\pi\otimes\varphi)_{F_2}}\\
        &= (\dim\pi)\,\itwo{\varphi}&\\
        &>(\dim\pi)\,\ione{\varphi}&\\
        &=\ione{\pi\otimes\varphi}=\iw{(\pi\otimes\varphi)_{F(\pi\otimes\varphi)}}.
    \end{align*} 
    
    By \cite[Theorem 4.6]{LaghribiGenericSplitting}, we know that the possible Witt indices of $\pi\otimes\varphi$ are the ones occuring in the standard splitting tower $K_0=F, K_1, K_2,\ldots$ of $\pi\otimes\varphi$.
    The inequality $\iw{(\pi\otimes\varphi)_{F_2}} > \iw{(\pi\otimes\varphi)_{F(\pi\otimes\varphi)}}=\iw{(\pi\otimes\varphi)_{K_1}}$ therefore yields 
    \begin{align*}(\dim\pi)\,\itwo{\varphi}=\iw{(\pi\otimes\varphi)_{F_2}}
    \geq \iw{(\pi\otimes\varphi)_{K_2}}.
    \end{align*}
    Further, since $h_\nd(\pi\otimes\varphi)\geq 2$, we know $\iw{(\pi\otimes\varphi)_{K_2}}=\itwo{\pi\otimes\varphi}$
    and the assertion follows.
\end{proof}

\new{Another condition, under which we can prove that the equality $\ione{\pi\otimes\varphi} = (\dim\pi)\,\ione{\varphi}$ holds, is assuming that the reverse implication of \Cref{Th:IsotropyPfisterMultiples} holds.}

\begin{proposition}\label{prop:SuffConditionForEqualityForIOne}
    Let $\varphi$ be a form of dimension at least two such that $\pi\otimes\varphi$ is anisotropic over $F$, where $\pi$ is similar to a bilinear Pfister form. 
    For all forms $\psi$ over $F$ with $\pi\otimes\psi$ anisotropic over $F$, assume that $(\pi\otimes\psi)_{F(\pi\otimes\varphi)}$ isotropic implies that $\psi$ is isotropic over $F(\varphi)$. 
    Then $\ione{\pi\otimes\varphi} = (\dim\pi)\,\ione{\varphi}$.
\end{proposition}
\begin{proof}
    By \Cref{Th:IoneIneq}, we already know the inequality $\ione{\pi\otimes\varphi}\geq(\dim\pi)\ione{\varphi}$.
    To complete the proof, we show that if this inequality is strict, we consider any form $\psi\preccurlyeq \varphi$ of codimension $\ione{\varphi}$ and show that $\pi\otimes\psi$ and $\psi_{F(\varphi)}$ are anisotropic, but $(\pi\otimes\psi)_{F(\pi\otimes\varphi)}$ is isotropic.
    
    Since we have $\dim\psi=\dim\varphi-\ione{\varphi}$, we clearly have $\dim\psi-\ione{\psi}<\dim\varphi-\ione{\varphi}$ and thus, \Cref{IoneIsotropyEquivalence} implies that $\psi_{F(\varphi)}$ is anisotropic.
    But on the other hand, we have $\pi\otimes\psi\preccurlyeq\pi\otimes\varphi$ and 
    \[\dim(\pi\otimes\psi)=\dim(\pi\otimes\varphi)-\dim(\pi)\,\ione{\varphi}> \dim(\pi\otimes\varphi)-\ione{\pi\otimes\varphi}\]
    and so, \Cref{Lem:SubFormOfIsotropicFormIsIsotropic} implies $(\pi\otimes\psi)_{F(\pi\otimes\varphi)}$ to be isotropic. 
\end{proof}

As noted in \Cref{Sec:Intro}, the isotropy of $(\pi\otimes\psi)_{F(\pi\otimes\varphi)}$ does not generally imply that $\psi_{F(\varphi)}$ is isotropic. 
In the following proposition, we impose some additional assumptions on $\varphi$, $\psi$ and $\pi$, under which the implication holds.

\begin{proposition}\label{prop:2.13}
    Let $\varphi$ and $\psi$ be forms of dimension at least two and let $\pi$ be a bilinear Pfister form such that $\pi\otimes\varphi$ and $\pi\otimes\psi$ are anisotropic over $F$. 
    Suppose that $\ione{\pi \otimes \varphi} = (\dim\pi)\,\ione{\varphi}$ and that $\varphi_{F(\psi)}$ and $(\pi\otimes\psi)_{F(\pi\otimes\varphi)}$ are isotropic. 
    Then $\psi_{F(\varphi)}$ is isotropic.
\end{proposition}
\begin{proof}
    Since $(\pi\otimes\psi)_{F(\pi\otimes\varphi)}$ is isotropic, \Cref{IoneIsotropyEquivalence} and the assumption imply 
    \begin{align*}    
        \dim(\pi\otimes\psi)-\ione{\pi\otimes\psi}&\geq\dim(\pi\otimes\varphi)-\ione{\pi\otimes\varphi}\\
        &=\dim(\pi\otimes\varphi)-(\dim\pi)\,\ione{\varphi}\\
        &=(\dim\pi)\,(\dim\varphi-\ione{\varphi}).
    \end{align*}
    Rearranging terms and applying \Cref{Th:IoneIneq} now yields 
    \[(\dim\pi)\,\ione{\psi}\leq \ione{\pi\otimes\psi}\leq (\dim\pi)\, (\dim\psi-\dim\varphi +\ione{\varphi}).\]
    Cancelling $\dim\pi$ and rearranging again yields 
    \[\dim\psi-\ione{\psi}\geq \dim\varphi-\ione{\varphi}.\]
    
    On the other hand, since $\varphi_{F(\psi)}$ is isotropic, \Cref{IoneIsotropyEquivalence} further implies 
    \[\dim\varphi-\ione{\varphi}\geq \dim\psi-\ione{\psi}.\]
    We thus have an equality $\dim\varphi-\ione{\varphi}= \dim\psi-\ione{\psi}$ which shows that $\psi_{F(\varphi)}$ is isotropic by \Cref{IoneIsotropyEquivalence} again.
\end{proof}

\subsection*{Maximal Splitting}

An easy consequence of the separation theorem \ref{thm:HoffmannSeparation} is that if $\dim\varphi=2^n+m$ for some $0<m\leq 2^n$, we have $\ione{\varphi}\leq m$, see \cite[Lemma 4.1]{HL06}.
Thus $\varphi$ is said to have \emph{maximal splitting}, if $\ione{\varphi}=m$.
The following proposition shows that maximal splitting of a quadratic form is handed down to its anisotropic Pfister multiples.

\begin{corollary}\label{prop:MaximalSplitting}
    Let $\varphi$ be a form of dimension at least three with maximal splitting.
    Let further $\pi$ be similar to a bilinear Pfister form such that $\pi\otimes\varphi$ is anisotropic over $F$. 
    Then $\pi\otimes\varphi$ has maximal splitting; so in particular, we have $\ione{\pi\otimes\varphi} = (\dim\pi)\,\ione{\varphi}$.
\end{corollary}
\begin{proof}
    We write $\dim\varphi=2^n+k$ with $1\leq k\leq2^n$.
    Since $\varphi$ has maximal splitting, we have $\ione{\varphi}=k$.
    By \Cref{Th:IoneIneq}, we have
    \[\ione{\pi\otimes\varphi}\geq k\dim\pi\]
    and so, $\pi\otimes\varphi$ has maximal splitting.
\end{proof}

\new{The maximal splitting of $\varphi$ is also inherited by multiples $\psi\otimes\varphi$ where $\psi$ is not a bilinear Pfister form itself, but a subform of one of a sufficiently large dimension.}

\begin{corollary}\label{prop:MaxSplittingPfisterNeighbor}
    Let $\varphi$ be a form of dimension at least two and $\pi$ be similar to a bilinear Pfister form such that $\pi\otimes\varphi$ is anisotropic over F. 
    Let $\psi$ be a subform of $\pi$ such that $\dim\psi > \dim \pi - \frac{(\dim \pi)\,\ione{\varphi}}{\dim\varphi}$. 
    If $\varphi$ has maximal splitting, then $\psi\otimes\varphi$ also has maximal splitting.
    In particular, we have
    \[\ione{\psi \otimes \varphi} = \dim\pi\,\ione{\varphi}  - \dim \varphi\,(\dim \pi - \dim \psi).\]
\end{corollary}
\begin{proof}
    By the condition on $\dim\psi$, we clearly have 
    \[\dim(\psi\otimes\varphi)>\dim(\pi\otimes\varphi)-(\dim\pi)\,\ione{\varphi}.\]
    Thus, \Cref{cor:BigSubformsOfProductsBecomeIsotropic} implies that $\psi\otimes\varphi$ is isotropic over $F(\pi\otimes\varphi)$.
    Clearly, the form $(\pi\otimes\varphi)_{F(\psi\otimes\varphi)}$ is isotropic as well and so, we have
    \[\dim(\psi\otimes\varphi)-\ione{\psi\otimes\varphi}=\dim(\pi\otimes\varphi)-\ione{\pi\otimes\varphi}\]
    by \Cref{IoneIsotropyEquivalence}.
    Since $\pi\otimes\varphi$ has maximal splitting by \Cref{prop:MaximalSplitting} so has $\psi\otimes\varphi$ and from $\ione{\pi\otimes\varphi}=(\dim\pi)\,\ione{\varphi}$, we obtain
    \[\begin{aligned}    \ione{\psi\otimes\varphi} &=\dim\pi\,\ione{\varphi}+\dim(\psi\otimes\varphi) -\dim(\pi\otimes\varphi) \\
    &=\dim\pi\,\ione{\varphi}  - \dim \varphi\,(\dim \pi - \dim \psi). 
    \end{aligned} \qedhere\]
\end{proof}

\section{Isotropy Indices over Function Fields of Quadrics}\label{sec:IndexOverFunctionFields}

\new{
The goal of this section is to show that
\begin{equation}\label{Eq:MainGoal}
    \istar{(\pi\otimes\varphi)_{F(\pi\otimes\psi)}} \geq(\dim\pi)\,\istar{\varphi_{F(\psi)}}
\end{equation}
for $\istarempty\in\{\iwempty,\iqlempty\}$. Note that if we could prove it without any restrictions on $\varphi$ and $\psi$, then \eqref{Eq:MainGoal} would also hold for $\istarempty=\itiempty$. Unfortunately, to show the inequality for the Witt index, we will need to assume that $\varphi$ is nonsingular and $\psi$ is not totally singular (see \Cref{Cor:DefectsWittProducts}). This restriction essentially arises from  parts \ref{Prop:IsotropySubforms2} and \ref{Prop:IsotropySubforms3} of \Cref{Prop:IsotropySubforms}, or rather, of the lack of their generalization to singular forms. However, we do not need to impose any restrictions for the results about the defect. That is because for any quadratic form $\varphi$, we have $\iql{\varphi}=\iql{\ql{\varphi}}$.
}

\subsection*{Estimates via a two-dimensional form} 

Consider two quadratic forms $\varphi$, $\psi$, and $\istarempty\in\{\iwempty,\iqlempty,\itiempty\}$. Computing $\istar{\varphi_{F(\psi)}}$ is difficult in general; however, recall that by \Cref{Prop:IsotropySubforms}, a lot is known if the extension $F(\psi)/F$ is quadratic. Therefore, an essential step in our approach is an estimation of $\istar{\varphi_{F(\psi)}}$ by $\istar{\varphi_{F(\mu)}}$ with a two-dimensional form $\mu\preccurlyeq\psi$. \new{In the cases of characteristic other than 2 and nonsingular quadratic forms in characteristic 2, this follows from \cite[Prop.~22.18]{ElmanKarpenkoMerkurjev2008}. In characteristic 2, we were not able to reach the full generality; we are missing the case of the Witt index when $\psi$ is totally singular. However, it does not affect our results, because (as discussed above) the main missing piece comes from \Cref{Prop:IsotropySubforms}.}

\begin{lemma}\label{Lemma:TransDefect}
Let $\varphi$, $\psi$ and $\mu$ be quadratic forms over $F$ with $\mu\preccurlyeq\psi$ and $\dim\mu\geq2$. Then $\iql{\varphi_{F(\mu)}}\geq\iql{\varphi_{F(\psi)}}$. 
\end{lemma}

\begin{proof}
We may assume that $\varphi$ is anisotropic. Moreover, recall that for any field extension $E/F$, we have ${\iql{\varphi_E}=\iql{\ql{\varphi}_E}}$; therefore, without loss of generality, let $\varphi$ be totally singular.

If $\psi$ is not totally singular, then $F(\psi)/F$ is separable and thus, $\varphi$ is anisotropic over $F(\psi)$, i.e., $\iql{\varphi_{F(\psi)}}=0$, and the claim is trivial. 

If $\mu$ is not totally singular, then it follows that $\psi$ is not totally singular, and we are in the previous case.

Now, assume that $\varphi$, $\psi$ and $\mu$ are all totally singular, so in particular $\mu\subseteq\psi$.
The assertion then follows from \cite[Lemma~3.4]{Scu16-Split} which can be applied because of \cite[Lemma~4.5]{Scu16-Split}.
\end{proof}

\begin{lemma} \label{Lemma:TransDefect_phi=psi}
Let $\varphi$ and $\mu$ be quadratic forms over $F$ such that $\mu\preccurlyeq\varphi$ and $\dim\mu\geq2$. Then we have inequalities 
\begin{align*}
\iw{\varphi_{F(\mu)}}&\geq\iw{\varphi_{F(\varphi)}},\\
\iql{\varphi_{F(\mu)}}&\geq\iql{\varphi_{F(\varphi)}},\\
\iti{\varphi_{F(\mu)}}&\geq\iti{\varphi_{F(\varphi)}}.
\end{align*}  
\end{lemma}

\begin{proof}
The inequality of Witt indices follows from \Cref{cor:FirstWittIndexMinimal}, and the claim about the defects is a direct consequence of \Cref{Lemma:TransDefect}. 
\end{proof}

\begin{lemma} \label{Lemma:TransIsotropyIndeces}
    Let $\varphi$, $\psi$ and $\mu$ be quadratic forms over $F$ such that $\mu\preccurlyeq\psi$ and $\dim\mu\geq2$. Moreover, suppose that $\psi$ is not totally singular. Then we have inequalities
    \begin{align*}
        \iw{\varphi_{F(\mu)}}&\geq \iw{\varphi_{F(\psi)}},\\
        \iql{\varphi_{F(\mu)}}&\geq \iql{\varphi_{F(\psi)}},\\
        \iti{\varphi_{F(\mu)}}&\geq \iti{\varphi_{F(\psi)}}.
    \end{align*}
\end{lemma}

\begin{proof}
    First, suppose that $\psi_{F(\mu)}$ is reducible as a polynomial; recall that it happens if and only if $(\psi_{F(\mu)})_\nd\cong\H$ or $\dim((\psi_{F(\mu)})_\nd)\leq1$. The latter case is impossible, because the nondefective part of a form that is not totally singular cannot be totally singular.  For the former case, note that ${(\psi_{F(\mu)})_\nd\cong\H}$ necessarily means that $\ql{\psi}\cong s\times\sqf{0}$ for some $s\geq0$ (because otherwise $\dim\ql{(\psi_{F(\mu)})_\nd}\geq1$). 
    Therefore, $(\psi_\nd)_{F(\mu)}\cong\H$, and so $\psi_\nd\cong c\mu$ for some $c\in F^*$ by the Subform Theorem (see \cite[Th.~22.5]{ElmanKarpenkoMerkurjev2008}). 
    Since the field extension $F(\psi)/F(\psi_\nd)$ is trivial or purely transcendental by \Cref{Lemma:FunFieldPurTransc}, it follows that 
    \[\iw{\varphi_{F(\mu)}}=\iw{\varphi_{F(\psi_\nd)}}=\iw{\varphi_{F(\psi)}},\]
    and analogously for the defect and the total isotropy index. 

    Now suppose that $\psi_{F(\mu)}$ is irreducible; then the field $F(\mu)(\psi)$ is well-defined.
    Since $\iw{\psi_{F(\mu)}}\geq\iw{\psi_{F(\psi)}}$ by \Cref{Lemma:TransDefect_phi=psi}, and $\iw{\psi_{F(\psi)}}=\ione{\psi}>0$ as $\psi$ is not totally singular, we get $\iw{\psi_{F(\mu)}}>0$. Hence, the field $F(\mu)(\psi)$ is a purely transcendental  extension of $F(\mu)$ by \Cref{Lemma:FunFieldPurTransc}. 
    Together with $F(\psi)\subseteq F(\mu)(\psi)$, it follows that we have
    \[\iw{\varphi_{F(\mu)}}=\iw{\varphi_{F(\mu)(\psi)}}\geq \iw{\varphi_{F(\psi)}}.\]
    Since an analogous inequality also holds for the defect by \Cref{Lemma:TransDefect} (even without further assumptions on $\psi$), the assertion about the total isotropy index follows.
\end{proof}

\subsection*{Reduction to anisotropic forms}
We present two technical lemmas that will help us reduce the main results to the case of anisotropic quadratic forms.

\begin{lemma} \label{Lemma:varphiWLOGanisotropic}
Let $\varphi$ be a quadratic form over $F$, and $\pi$ be a bilinear form over $F$. Let $E/F$ and $L_1/E$, $L_2/E$ be field extensions. Assume that $\istarempty\in\{\iwempty, \iqlempty, \itiempty\}$ and $\Square\in\{=,\leq,\geq,<,>\}$. If
\[\istar{(\pi\otimes(\varphi_E)_\an)_{L_1}}\:\Square\:(\dim\pi)\: \istar{((\varphi_E)_\an)_{L_2}},\]
then
\[\istar{(\pi\otimes\varphi)_{L_1}}\:\Square\:(\dim\pi)\: \istar{\varphi_{L_2}}.\]    
\end{lemma}

\begin{proof}
Let $\varphi_E\cong i\times\H\perp\varphi'\perp j\times\sqf{0}$ with $\varphi'$ anisotropic. We get
\begin{align*}
(\pi\otimes\varphi)_{L_1}&\cong (i\dim\pi)\times\H\perp(\pi\otimes\varphi')_{L_1}\perp (j\dim\pi)\times\sqf{0}, 
\quad \text{and} \\
\varphi_{L_2}&\cong i\times\H\perp\varphi'_{L_2}\perp j\times\sqf{0}.
\end{align*}
Therefore,
\begin{align*}
\iw{\pi\otimes\varphi}_{L_1}&=i\dim\pi+\iw{(\pi\otimes\varphi')_{L_1}}
\quad \text{and} \\
(\dim\pi)\:\iw{\varphi_{L_2}}&=i\dim\pi+(\dim\pi)\:\iw{\varphi'_{L_2}},
\end{align*}
as well as 
\begin{align*}
\iql{\pi\otimes\varphi}_{L_1}&=j\dim\pi+\iql{(\pi\otimes\varphi')_{L_1}}
\quad \text{and} \\
(\dim\pi)\:\iql{\varphi_{L_2}}&=j\dim\pi+(\dim\pi)\:\iql{\varphi'_{L_2}}.
\end{align*}
Thus, if it holds that
\[\istar{(\pi\otimes\varphi')_{L_1}}\:\Square\:(\dim\pi)\: \istar{\varphi'_{L_2}},\]
then we have
\[\istar{(\pi\otimes\varphi)_{L_1}}\:\Square\:(\dim\pi)\: \istar{\varphi_{L_2}}. \qedhere\] 
\end{proof}

\begin{lemma}\label{Lemma:psiWLOGanisotropic}
    Let $\varphi$ and $\psi$ be quadratic forms over $F$ and $\pi$ be a bilinear form over $F$. Let $L/F$ be a field extension such that $(\pi\otimes\psi)_L$ is irreducible. Assume that $\istarempty\in\{\iwempty, \iqlempty, \itiempty\}$ and $\Square\in\{=,\leq,\geq,<,>\}$. If
    \[\istar{(\pi\otimes\varphi)_{L(\pi\otimes\psi_\nd)}}\:\Square\:(\dim\pi)\:\istar{\varphi_{F(\psi_\nd)}},\]
    then
    \[\istar{(\pi\otimes\varphi)_{L(\pi\otimes\psi)}}\:\Square\:(\dim\pi)\:\istar{\varphi_{F(\psi)}}.\]
    Moreover, if $\iw{\psi}>0$, then
    \[\istar{(\pi\otimes\varphi)_{L(\pi\otimes\psi)}}\geq(\dim\pi)\:\istar{\varphi_{F(\psi)}}\]
    with equality if $\pi\otimes\varphi_{\an}$ is anisotropic.
\end{lemma}

\begin{proof}
The field extensions $F(\pi\otimes\psi)/F(\pi\otimes\psi_\nd)$ and $F(\psi)/F(\psi_\nd)$ are purely transcendental by \Cref{Lemma:FunFieldPurTransc}. 
Thus, we have 
\[ \istar{(\pi\otimes\varphi)_{F(\pi\otimes\psi)}}=\istar{(\pi\otimes\varphi)_{F(\pi\otimes\psi_\nd)}}
\quad \text{and} \quad
\istar{\varphi_{F(\psi)}}=\istar{\varphi_{F(\psi_\nd)}},
\]
and the first claim follows.

If $\iw{\psi}>0$, then the field extensions $F(\pi\otimes\psi)/F$ and $F(\psi)/F$ are purely transcendental (again by \Cref{Lemma:FunFieldPurTransc}), and hence
\[ \istar{(\pi\otimes\varphi)_{F(\pi\otimes\psi)}}=\istar{\pi\otimes\varphi}
\quad \text{and} \quad
\istar{\varphi_{F(\psi)}}=\istar{\varphi}.
\]
Note that $\istar{\pi\otimes\varphi}\geq(\dim\pi)\:\istar{\varphi}$. Moreover, by \Cref{Prop:DivideFactor}, if $\pi\otimes\varphi_\an$ is anisotropic, then even $\istar{\pi\otimes\varphi}=(\dim\pi)\:\istar{\varphi}$. Thus, the claim follows.
\end{proof}

\subsection*{Main Result}

The \new{key} step is choosing a two-dimensional form $\mu\preccurlyeq\psi$ in such way that we can compute the value of $\istar{\pi\otimes\mu}$. 

\begin{lemma} \label{Lemma:PiBetaIsotropic}
Let $\psi$ be an anisotropic quadratic form over $F$ and $\pi$ be an anisotropic 1-fold bilinear Pfister form over $F$. Assume that $\pi\otimes\psi$ is isotropic. Then there exists an anisotropic binary quadratic form $\mu$ over $F$ such that $\mu\preccurlyeq\psi$ and
\begin{enumerate}
    \item $\mu$ is nonsingular and $\iw{\pi\otimes\mu}=2$, or
    \item $\mu$ is totally singular and $\iql{\pi\otimes\mu}=2$.
\end{enumerate}
In particular, if $\iql{\pi\otimes\psi}>0$, then we can find $\mu$ satisfying part (ii).
\end{lemma}

\begin{proof}
Let $V$ be the $F$-vector space of $\psi$ and $\pi\cong\bif{1,a}$; as $\pi$ is anisotropic, we have $a\notin \squares{F}$. 
Since $\pi\otimes\psi\cong\psi\perp a\psi$ is isotropic, it follows that there exist nonzero vectors $u,v\in V$ such that $\psi(u)+a\psi(v)=0$. We need to distinguish between two cases:

First, assume that $\bb_\psi(u,v)\neq0$; then $u,v$ are linearly independent (because $\bb_\psi(u,u)=0$). Let $c=\bb_\psi(u,v)$. Set $x=\psi(u)$ and $y=\psi(c^{-1}v)$. As $\psi$ is anisotropic, we have $x,y\neq0$. Let $\mu=[x,y]$; then $\mu\preccurlyeq\psi$ (because $\mu\cong\psi|_{\spn\{u,v\}}$). Moreover, it follows from the equality $\psi(u)=a\psi(v)$ that $x=ac^2y$. Thus, we have
\[\pi\otimes\mu\cong[x,y]\perp a[x,y]\cong[x,y]\perp[a^{-1}c^{-2}x,ac^2y]\cong[x,y]\perp[y,x]\cong2\times\H,\]
and hence $\iw{\pi\otimes\mu}=2$.

Second, suppose $\bb_\psi(u,v)=0$. In this case, set $x=\psi(u)$ and $y=\psi(v)$; then $x=ay$, and $x$ and $y$ are nonzero. Let $\mu\cong\sqf{x,y}$; as $xy=ay^2\notin \squares{F}$, we have that $\mu$ is anisotropic. 
It follows that $\mu\preccurlyeq\psi$ and 
\[\pi\otimes\mu\cong\bif{1,a}\otimes\sqf{x,y}\cong\sqf{x,ax,y,ay}\cong\sqf{x,y}\perp2\times\sqf{0};\]
therefore, $\iql{\pi\otimes\mu}=2$.

If $\iql{\pi\otimes\psi}>0$, we may exchange $\psi$ for $\ql{\psi}$. Then $\mu$ must be totally singular, i.e., we are in the second case.
\end{proof}

\new{Using $\mu$ from \Cref{Lemma:PiBetaIsotropic}, we can compare $\istar{(\pi\otimes\varphi)_E}$ with $\istar{\varphi_{E(\mu)}}$ (for an appropriate choice of the field $E$), and consequently also with $\istar{\varphi_{F(\psi)}}$. \Cref{Lemma:DefectsProducts} covers the case of $\istarempty=\iqlempty$, and \Cref{Lemma:WittProducts} deals with $\istarempty=\iwempty$.}

\begin{lemma} \label{Lemma:DefectsProducts} 
Let $\varphi$ and $\psi$ be quadratic forms over $F$ of dimension at least two, and let $\pi$ be an 1-fold bilinear Pfister form over $F$. Let $E/F$ be a field extension such that $\pi_E$ and $\psi_E$ are anisotropic but $\iql{(\pi\otimes\psi)_E}>0$. 
\begin{enumerate}
    \item\label{Lemma:DefectsProducts1} For some $\mu\preccurlyeq\psi_E$ with $\dim\mu=2$, we have that
    \[\iql{(\pi\otimes\varphi)_E}\geq2\:\iql{\varphi_{E(\mu)}}.\]
    \item\label{Lemma:DefectsProducts2} We have that
    \[\iql{(\pi\otimes\varphi)_E}\geq2\:\iql{\varphi_{F(\psi)}}.\]
\end{enumerate}
\end{lemma}

\begin{proof}
First, note that since $\ql{(\pi\otimes\varphi)_E}\cong\pi_E\otimes\ql{\varphi}_E$, we can assume without loss of generality that $\varphi$ is totally singular. By \Cref{Lemma:varphiWLOGanisotropic} (with $L_1:=E$ and $L_2:=E(\mu)$), it is sufficient to prove the claim in the case when $\varphi_E$ is anisotropic.

\medskip
\ref{Lemma:DefectsProducts1} 
Since $\iql{(\pi\otimes\psi)_E}>0$, by \Cref{Lemma:PiBetaIsotropic}, we can find a binary quadratic form $\mu\preccurlyeq\psi_E$ such that $\iql{\pi_E\otimes\mu}=2$.
If $\varphi_{E(\mu)}$ is anisotropic, then the claim is trivial. 
Hence, suppose that $\varphi_{E(\mu)}$ is isotropic; by \Cref{Prop:IsotropySubforms}\ref{Prop:IsotropySubforms1} (recall that $\varphi_E$ is totally singular), there exists a totally singular quadratic form $\tau$ over $E$ such that $\dim\tau=\iql{\varphi_{E(\mu)}}$ and $\mu\otimes\tau\subseteq\varphi_E$. 
Then
\[\pi_E\otimes\mu\otimes\tau\subseteq(\pi\otimes\varphi)_E;\]
therefore,
\[\iql{(\pi\otimes\varphi)_E}
    \geq\iql{\pi_E\otimes\mu\otimes\tau}
    \geq\iql{\pi_E\otimes\mu}\dim\tau
    =2\dim\tau
    =2\iql{\varphi_{E(\mu)}}.
    \]

\medskip
\ref{Lemma:DefectsProducts2} Let $\mu$ be the same quadratic form as in \ref{Lemma:DefectsProducts1}; then $\iql{\varphi_{E(\mu)}}\geq\iql{\varphi_{E(\psi)}}$ by \Cref{Lemma:TransDefect}. 
Moreover, note that we clearly have $\iql{\varphi_{E(\psi)}}\geq\iql{\varphi_{F(\psi)}}$.
Therefore, $\iql{\varphi_{E(\mu)}}\geq\iql{\varphi_{F(\psi)}}$, and part \ref{Lemma:DefectsProducts2} is just a consequence of part \ref{Lemma:DefectsProducts1}.
\end{proof}

\begin{lemma} \label{Lemma:WittProducts}
Let $\varphi$ and $\psi$ be  quadratic forms over $F$ of dimension at least two with $\varphi$ nonsingular and $\psi$ not totally singular. Let $\pi$ be an 1-fold bilinear Pfister form over $F$. Let $E/F$ be a field extension such that $\pi_E$ and $\psi_E$ are anisotropic but $(\pi\otimes\psi)_E$ is isotropic.
\begin{enumerate}
    \item\label{Lemma:WittProducts1} For some $\mu\preccurlyeq\psi_E$ with $\dim\mu=2$, we have that
    \[\iw{(\pi\otimes\varphi)_E}\geq2\:\iw{\varphi_{E(\mu)}}.\]
    \item\label{Lemma:WittProducts2} We have that
    \[\iw{(\pi\otimes\varphi)_E}\geq2\:\iw{\varphi_{F(\psi)}}.\]
\end{enumerate}
\end{lemma}

\begin{proof}
By \Cref{Lemma:varphiWLOGanisotropic} (with $L_1:=E$ and $L_2:=E(\mu)$), we may assume that $\varphi_E$ is anisotropic.

\medskip
\ref{Lemma:WittProducts1} 
Since $\pi_E$ and $\psi_E$ are anisotropic but $(\pi\otimes\psi)_E$ is isotropic, we can apply \Cref{Lemma:PiBetaIsotropic} to find a binary quadratic form $\mu\preccurlyeq\psi_E$ such that either $\iw{\pi_E\otimes\mu}=2$ (if $\mu$ is nonsingular), or $\iql{\pi_E\otimes\mu}=2$ (if $\mu$ is totally singular). In any case, if $\iw{\varphi_{E(\mu)}}=0$, then the claim is trivial; therefore, assume that $r=\iw{\varphi_{E(\mu)}}>0$.

First, assume that $\mu$ is nonsingular. Without loss of generality, we may assume that $\mu\cong[1,d]$ for some $d\in E^*$. By part \ref{Prop:IsotropySubforms3} of \Cref{Prop:IsotropySubforms}, we can find a bilinear form $\tau$ over $E$ of dimension $r$ such that $\tau\otimes[1,d]\subseteq\varphi_E$. Then \[\tau\otimes\pi_E\otimes\mu\cong\pi_E\otimes\tau\otimes\mu\subseteq(\pi\otimes\varphi)_E,\] 
and hence
\[\iw{(\pi\otimes\varphi)_E}\geq\iw{\tau\otimes\pi_E\otimes\mu}\geq(\dim\tau)\iw{\pi_E\otimes\mu}=2r=2\iw{\varphi_{E(\mu)}}.\]

Now, let $\mu$ be totally singular. 
Note that in this case, it follows from the proof of \Cref{Lemma:PiBetaIsotropic} that $\mu\cong x\sqf{1,d}$ 
for some $x\in E^*$ and $d\in E$ such that ${\pi_E\cong\bif{1,d}}$; without loss of generality, assume that $\mu\cong\sqf{1,d}$.
Invoking part \ref{Prop:IsotropySubforms2}  of \Cref{Prop:IsotropySubforms}, we can find $a_i, b_i, c_i\in E^*$, $1\leq i\leq r$, such that
\[\bigperp_{i=1}^r([a_i,b_i]\perp d[a_i,c_i])\subseteq\varphi_E.\]
Note that for any $1\leq i\leq r$, we have
\begin{align*}
\pi_E\otimes([a_i,b_i]\perp d[a_i,c_i]) 
&\cong \sqf{1,d}\otimes([a_i,b_i]\perp d[a_i,c_i]) \\
&\cong [a_i,b_i]\perp d[a_i,c_i] \perp d[a_i,b_i]\perp d^2[a_i,c_i]\\
&\cong \sqf{1,d}\otimes([a_i,b_i]\perp [a_i,c_i])\\
&\cong \sqf{1,d}\otimes([a_i,b_i+c_i]\perp\H)\\
&\cong \sqf{1,d}\otimes [a_i,b_i+c_i] \perp 2\times\H.
\end{align*}
As
\[\pi_E\otimes\left(\bigperp_{i=1}^r([a_i,b_i]\perp d[a_i,c_i])\right)\subseteq(\pi\otimes\varphi)_E,\]
it follows that
\[\iw{(\pi\otimes\varphi)_E}\geq2r=2\iw{\varphi_{E(\mu)}}.\]

\medskip
\ref{Lemma:WittProducts2} Let $\mu$ be as in part \ref{Lemma:WittProducts1}. 
Applying \Cref{Lemma:TransIsotropyIndeces}, we get the inequality ${\iw{\varphi_{E(\mu)}}\geq\iw{\varphi_{E(\psi)}}}$.
Moreover, we have $\iw{\varphi_{E(\psi)}}\geq\iw{\varphi_{F(\psi)}}$. Invoking part \ref{Lemma:WittProducts1}, we get that
\[
\iw{(\pi\otimes\varphi)_E}
\geq2\iw{\varphi_{E(\mu)}}
\geq2\iw{\varphi_{E(\psi)}}
\geq2\iw{\varphi_{F(\psi)}}.
\qedhere
\]
\end{proof}

\begin{remark} \label{Rem:ProductsLemmasApplication}
Note that for any field extension $K/F$ such that $\pi_K$ and $\psi_K$ are anisotropic (so in particular for $K=F$ and $K=\FX$), the field $E=K(\pi\otimes\psi)$ satisfies the assumptions of  \Cref{Lemma:DefectsProducts} and \Cref{Lemma:WittProducts}. In that case, $\pi_E$ and $\psi_E$ are anisotropic by Hoffmann's separation theorem \ref{thm:HoffmannSeparation}, while $(\pi\otimes\psi)_E$ is isotropic.
\end{remark}

\begin{theorem} \label{Cor:DefectsWittProducts}
Let $\varphi$ and $\psi$ be quadratic forms over $F$ of dimension at least two, and let $\pi$ be a bilinear Pfister form. 
\begin{enumerate}
\item \label{Cor:DefectsWittProducts1} We have 
    \[\iql{(\pi\otimes\varphi)_{F(\pi\otimes\psi)}}\geq (\dim\pi)\:\iql{\varphi_{F(\psi)}}.\]
\item \label{Cor:DefectsWittProducts2} Moreover, if $\varphi$ is nonsingular and $\psi$ is not totally singular, then we also have
    \[\iw{(\pi\otimes\varphi)_{F(\pi\otimes\psi)}}\geq (\dim\pi)\:\iw{\varphi_{F(\psi)}}.\]
\end{enumerate}
\end{theorem}

\begin{proof}
First, note that by \Cref{Lemma:psiWLOGanisotropic}, we can assume that $\psi$ is anisotropic.

\medskip
\ref{Cor:DefectsWittProducts1}  Let $n$ be such that $\pi$ is an $n$-fold bilinear Pfister form. For $n=1$, the assumptions of \Cref{Lemma:DefectsProducts} with $E=F(\pi\otimes\psi)$ are satisfied by \Cref{Rem:ProductsLemmasApplication}, and the claim follows from part \ref{Lemma:DefectsProducts2} of that lemma. If $n>1$, then we proceed by induction: write $\pi\cong\sqf{1,a}\otimes \pi'$ for some $a\in F^*$ and a bilinear Pfister form $\pi'$ over $F$. We have
\begin{align*}
\iql{(\pi\otimes\varphi)_{F(\pi\otimes\psi)}}
&=\iql{(\sqf{1,a}\otimes\pi'\otimes\varphi)_{F(\sqf{1,a}\otimes\pi'\otimes\psi)}}\\
&\geq 2\:\iql{(\pi'\otimes\varphi)_{F(\pi'\otimes\psi)}}\\
&\geq (2\dim\pi')\:\iql{\varphi_{F(\psi)}}\\
&=(\dim\pi)\:\iql{\varphi_{F(\psi)}},
\end{align*}
where the first inequality is by \Cref{Lemma:DefectsProducts}\ref{Lemma:DefectsProducts2} (this time with $\pi_0\cong\sqf{1,a}$ and $\psi_0\cong\pi'\otimes\psi$; the assumptions of the lemma are satisfied again by \Cref{Rem:ProductsLemmasApplication} with $E=F(\pi_0\otimes\psi_0)$), and the second inequality follows by the induction assumption.

\medskip
\ref{Cor:DefectsWittProducts2} The proof is completely analogous to the proof of \ref{Cor:DefectsWittProducts1}, using \Cref{Lemma:WittProducts} in place of \Cref{Lemma:DefectsProducts}.
\end{proof}

\section{Generic Pfister forms}

In this section, we focus on the so-called generic Pfister forms, i.e., bilinear Pfister forms $\pi\cong\bipf{X_1,\dots,X_n}$ with $X_1,\dots,X_n$ algebraically independent variables. \new{Note that \Cref{Prop:IsotropyXXXPFmultiples} suggests that generic Pfister forms are well-behaved in terms of preserving the isotropy. Indeed, the following theorem shows that in the case of generic Pfister forms, we get the equality $\istar{(\pi\otimes\varphi)_{K(\pi\otimes\psi)}}=(\dim\pi)\:\istar{\varphi_{F(\psi)}}$ for $\istarempty\in\{\iwempty,\iqlempty,\ioneempty\}$   (cf. \Cref{Th:IoneIneq} and \Cref{Cor:DefectsWittProducts}). However, analogously as in \Cref{sec:IndexOverFunctionFields}, we will need to impose some restrictions on the forms $\varphi$ and $\psi$ in the case of the Witt index.}

\begin{theorem}\label{Th:GenericPfisterMultiple}
Let $\varphi$ and $\psi$ be quadratic forms over $F$ of dimension at least two. 
Let $K=F\dbrac{X_1}\dots\dbrac{X_n}$ and $\pi\cong\bipf{X_1,\dots,X_n}$. 
\begin{enumerate}
     \item \label{Th:GenericPfisterMultiple1} If $\varphi$ is nonsingular and $\psi$ is not totally singular, then we have
            \[\iw{(\pi\otimes\varphi)_{K(\pi\otimes\psi)}}=(\dim\pi)\:\iw{\varphi_{F(\psi)}}.\]
    \item \label{Th:GenericPfisterMultiple2} For arbitrary quadratic forms $\varphi$ and $\psi$, we have
            \[\iql{(\pi\otimes\varphi)_{K(\pi\otimes\psi)}}=(\dim\pi)\:\iql{\varphi_{F(\psi)}}.\]
    \item \label{Th:GenericPfisterMultiple3}
    For an arbitrary quadratic form $\varphi$, we have
    \[\ione{\pi\otimes\varphi}=(\dim\pi)\:\ione{\varphi}.\]
\end{enumerate}
\end{theorem}

\begin{proof}
By \Cref{Lemma:varphiWLOGanisotropic} (with $E:=F$, $L_1:=K(\pi\otimes\psi)$ and $L_2:=F(\psi)$ for parts \ref{Th:GenericPfisterMultiple1} and \ref{Th:GenericPfisterMultiple2} , resp. with $L_1:=K(\pi\otimes\varphi)$ and $L_2:=F(\varphi)$ for part \ref{Th:GenericPfisterMultiple3})  we can assume that $\varphi$ is anisotropic. Moreover, since $\pi\otimes\varphi\cong\pi\otimes\varphi_\an$ is anisotropic by \Cref{Lemma:IndicesResidueForms}, we can use \Cref{Lemma:psiWLOGanisotropic} to reduce to the case when $\psi$ is anisotropic.

Note that it is sufficient to prove the theorem for $n=1$, because for $n>1$, we only need to iterate the argument. Therefore, let $\pi\cong\bipf{X}$ and $K=\FX$.

\medskip

\ref{Th:GenericPfisterMultiple1} Since $\psi\subseteq\bipf{X}\otimes\psi$, we can apply \Cref{Lemma:TransIsotropyIndeces} (note that $\bipf{X}\otimes\psi$ is not totally singular). We get 
\[\iw{(\bipf{X}\otimes\varphi)_{K(\bipf{X}\otimes\psi)}} \leq\iw{(\bipf{X}\otimes\varphi)_{K(\psi)}}.\]
Moreover, as $K(\psi)\subseteq F(\psi)\dbrac{X}$ and invoking \Cref{Lemma:IndicesResidueForms}, we have that
\[\iw{(\bipf{X}\otimes\varphi)_{K(\psi)}}\leq \iw{(\bipf{X}\otimes\varphi)_{F(\psi)\dbrac{X}}}=2\:\iw{\varphi_{F(\psi)}}.\]
Put together, we get
\begin{equation}\label{Eq:WittGenericPfisterMultiple-leq}
\iw{(\bipf{X}\otimes\varphi)_{K(\bipf{X}\otimes\psi)}}\leq2\:\iw{\varphi_{F(\psi)}}.    
\end{equation}

On the other hand, $\bipf{X}\otimes\psi$ is isotropic over $K(\bipf{X}\otimes\psi)$, while both $\bipf{X}$ and $\psi$ are anisotropic over $K(\bipf{X}\otimes\psi)$ by Hoffmann's separation theorem \ref{thm:HoffmannSeparation}. Therefore, part \ref{Lemma:WittProducts2} of \Cref{Lemma:WittProducts} implies that 
\[\iw{(\bipf{X}\otimes\varphi)_{K(\bipf{X}\otimes\psi)}}\geq2\:\iw{\varphi_{K(\psi)}}.\]
Moreover, as $F(\psi)\subseteq K(\psi)\subseteq F(\psi)\dbrac{X}$, we have 
\[\iw{\varphi_{F(\psi)}}\leq\iw{\varphi_{K(\psi)}}\leq\iw{\varphi_{F(\psi)\dbrac{X}}}=\iw{\varphi_{F(\psi)}};\]
therefore, $\iw{\varphi_{K(\psi)}}=\iw{\varphi_{F(\psi)}}$. It follows that
\begin{equation}\label{Eq:WittGenericPfisterMultiple-geq}
\iw{(\bipf{X}\otimes\varphi)_{K(\bipf{X}\otimes\psi)}}\geq2\:\iw{\varphi_{F(\psi)}}.    
\end{equation}
Combining \eqref{Eq:WittGenericPfisterMultiple-leq} and \eqref{Eq:WittGenericPfisterMultiple-geq}, the claim follows.

\bigskip

\ref{Th:GenericPfisterMultiple2} 
Note that $\iql{(\pi\otimes\varphi)_K}=2\:\iql{\varphi}$ by \Cref{Lemma:IndicesResidueForms}. If $\psi$ is not totally singular, then $\pi\otimes\psi$ is also not totally singular, and so the extensions $K(\pi\otimes\psi)/K$ and $F(\psi)/F$ are separable. Thus, we have 
\[\iql{(\pi\otimes\varphi)_{K(\pi\otimes\psi)}}
=\iql{(\pi\otimes\varphi)_K}
=2\:\iql{\varphi}
=2\:\iql{\varphi_{F(\psi)}},\] 
so the claim holds. Therefore, assume that $\psi$ is totally singular.

The proof is analogous to the proof of part \ref{Th:GenericPfisterMultiple1} with $\iqlempty$ in place of $\iwempty$. Instead of \Cref{Lemma:TransIsotropyIndeces}, we use \Cref{Lemma:TransDefect} (note that this lemma put no assumptions on $\bipf{X}\otimes\psi$). In place of \Cref{Lemma:WittProducts}, we apply \Cref{Lemma:DefectsProducts} (here we need $\psi$ to be totally singular, to ensure that the isotropy of $\bipf{X}\otimes\psi$ over $K(\bipf{X}\otimes\psi)$ implies $\iql{(\bipf{X}\otimes\psi)_{K(\bipf{X}\otimes\psi)}}>0$). The details are left to the reader.

\bigskip

\ref{Th:GenericPfisterMultiple3} The proof goes along the same lines as the proof of part \ref{Th:GenericPfisterMultiple1} with $\itiempty$ in place of $\iwempty$ and $\psi=\varphi$. In place of \Cref{Lemma:TransIsotropyIndeces} we use  \Cref{Lemma:TransDefect_phi=psi}, and instead of \Cref{Lemma:WittProducts} we apply  \Cref{Th:IoneIneq}. The details are left to the reader.
\end{proof}

There are further connections between a quadratic form and their product with a generic Pfister form apart from the above relations concerning the isotropy indices.
As an example, we show the following link between the respective sets of similarity factors.

\begin{proposition}\label{prop:GenericMultipleRound}
    Let $q$ be an anisotropic quadratic form over $F$ and consider the quadratic form $\varphi=\bipf{X}\otimes q=q\perp Xq$ over $F\dbrac{X}$. 
    Let further $a\in G_{F\dbrac X}(\varphi)\cap F$. 
    Then $a\in G_F(q)$.
    In particular, $\varphi$ is round if and only if $q$ is round.
\end{proposition}
\begin{proof}
    Since both the set of represented elements and the set of similarity factors are not influenced by the quadratic radical, we may assume $q$ and thus also $\varphi$ to be nondefective.
    Let $q$ be of type $(r,s)$, so that $\varphi$ is of type $(2r,2s)$.
    Because $a$ is a similarity factor of $\varphi$, the orthogonal sum $\varphi\perp a\varphi$ is quasi-hyperbolic, i.e., we have 
    \[\iw{\varphi\perp a\varphi}=4r~~\text{ and }~~\iql{\varphi\perp a\varphi}=2s.\]
    Using \Cref{Lemma:IndicesResidueForms}, we have
    \[\iw{q\perp aq}=\frac12\iw{(q\perp aq\perp X(q\perp aq)}=\frac12\iw{\varphi\perp a\varphi}=\frac12\cdot 4r=2r\]
    and similarly
    \[\iql{q\perp aq}=s.\]
    Now \cite[Proposition 3.11]{HL04} yields the existence of a form $\psi$ of dimension $2r+s$ over $F$ that is both dominated by $q$ and $aq$. 
    Clearly, because all three forms $q, aq, \psi$ have the same dimension, we obtain $q\cong\psi\cong aq$, i.e., $a\in G_F(q)$.
    For the roundness, one implication is the well known statement \cite[Proposition 9.8 (1)]{ElmanKarpenkoMerkurjev2008}, the other one follows from the above using the equality
    \[D_{F\dbrac X}(\varphi)=D_F(q)+ XD_F(q).\qedhere\]
\end{proof}

\section*{Acknowledgments}
The authors would like to thank Adam Chapman and Ahmed Laghribi for pointing out an easier proof of \Cref{prop:GenericMultipleRound}.

K.Z. acknowledges support from the Pacific Institute for the Mathematical Sciences and a partial support from the National Science and Engineering Research Council of Canada. The research and findings may not reflect those of these institutions. 


\printbibliography

\end{document}